\documentclass[12pt, a4paper, reqno]{amsart}
\usepackage{latexsym,amsmath, amssymb, eucal, amscd, amstext, enumerate, mathrsfs, yfonts, amsfonts,comment, verbatim, color}
\usepackage[plainpages=false,colorlinks,pdfpagemode=None,bookmarksopen,linkcolor=blue,citecolor=blue,urlcolor=blue]{hyperref}
\usepackage{cite}

\def\tto{\;{\lower 1pt \hbox{$\rightarrow$}}\kern -10pt
\hbox{\raise 2pt \hbox{$\rightarrow$}}\;}

\newtheorem{Theorem}{Theorem}[section]
\newtheorem{Proposition}[Theorem]{Proposition}

\theoremstyle{definition}

\theoremstyle{remark}
\newtheorem{Remark}[Theorem]{Remark}

\begin{document}

\title[Strongly Entanglement Breaking channels]{Characterizations of Strongly Entanglement Breaking channels for infinite-dimensional quantum systems}

\author[Muoi]{Bui Ngoc Muoi}
\address[Bui Ngoc Muoi]{Department of Applied Mathematics, The Hong Kong Polytechnic University, Hung Hom, Hong Kong, and Department of Mathematics, Hanoi Pedagogical University 2, Vinh Phuc, Vietnam.}
\email{\tt buingocmuoi@hpu2.edu.vn}

\author[Sze]{Nung-Sing Sze}
\address[Nung-Sing Sze]{Department of Applied Mathematics, The Hong Kong Polytechnic University, Hung Hom, Hong Kong.}
\email{\tt raymond.sze@polyu.edu.hk}

\thanks{Corresponding author: Nung-Sing Sze, E-mail: raymond.sze@polyu.edu.hk}


\begin{abstract}
Entanglement breaking (EB) channels, as completely positive and trace-preserving linear operators, sever the entanglement between the input system and other systems. In the realm of infinite-dimensional systems, a related concept known as strongly entanglement breaking (SEB) channels emerges. This paper delves into characterizations of SEB channels, delineating necessary and sufficient conditions for a channel to be classified as SEB, especially with respect to the commutativity of its range. Moreover, we demonstrate that every closed self-adjoint subspace of trace-zero operators, with the trace norm, is the null space of a SEB channel.
\end{abstract}



\maketitle

\section{Introduction}

Entanglement breaking (EB) channels are channels which disrupt entanglement with other quantum systems. 
Specifically, when an EB channel is applied to a component of a composite system, any entanglement in the input state transforms into separable states.
This concept was first introduced in \cite{HSR03} as an extension of the classical-quantum and quantum-classical channels considered in \cite{Holevo98}. The characterization of these channels has been the subject of intensive research for finite-dimensional systems, as seen in works such as \cite{BaHo24, HSR03, KLOPR21, KLPR22}. Some of these findings have been expanded to infinite-dimensional systems in \cite{DSS24, HSW05, He13, LD15}.

We recall some notions in quantum information theory, see e.g. \cite{Hou10,HSW05}.
Consider infinite-dimensional Hilbert spaces $H$ and $K$. 
Let $B(H)$ be the set of bounded linear operators on $H$, and let $B(H)^+$ denote the set of positive semi-definite operators in $B(H)$. Let $\mathcal{T}(H)$ be the space of trace-class operators on $H$. A state in $\mathcal{T}(H)$ is a self-adjoint positive and trace-one operator on $H$. Denote by $S(H)$  the set of states in $\mathcal{T}(H)$. A state $\rho\in S(H\otimes K)$ is called \textit{separable} if it is a limit in trace-norm of states of the form
\begin{equation}\label{sep-state}
	\rho=\sum_{i=1}^{n}p_ig_i\otimes h_i,\quad \mbox{ for some } g_i\in S(H), h_i\in S(K), p_i>0, \mbox{ with } \sum_{i=1}^n p_i=1.
\end{equation}

In particular, if $\rho\in S(H\otimes K)$ can be expressed  as an infinite sum of product states:
\begin{equation}\label{count-sepa}
	\rho=\sum_{i=1}^{\infty}p_ig_i\otimes h_i,\quad \mbox{ for some } g_i\in S(H), h_i\in S(K), p_i>0,  \mbox{ with } \sum_{i=1}^\infty p_i=1,
\end{equation}
then $\rho$ is said to be \textit{countably separable}, see \cite{HouChai17,HSW05}. 
For finite-dimensional systems, a separable state always has the form as outlined in \eqref{sep-state}. 
However, in infinite-dimensional systems, a separable state may not necessarily be countably separable, i.e. it may not have the form \eqref{count-sepa}, see, e.g. \cite{HSW05}.

A quantum channel $\Phi:\mathcal{T}(H)\to\mathcal{T}(K)$ is a completely positive and trace-preserving linear map. 
Specifically, a channel $\Phi$ is called \textit{entanglement breaking} (EB in short) if, 
for any Hilbert space $R$ and any state $\rho\in\mathcal{T}(R\otimes H)$, the output state $(I_R\otimes\Phi)(\rho)$ is always separable in $S(R\otimes K)$. 
Similarly, $\Phi$ is called \textit{strongly entanglement breaking} (SEB in short) if $(I_R\otimes\Phi)(\rho)$ is always countably separable.

The following provides us with the structure of SEB channels, which we will utilize throughout this paper. It shows that a SEB channel always has a measure-and-prepare form (call a Holevo's form) \eqref{holevo-SEB} and can be represented by operator-sum of rank-one operations \eqref{rank1rep}.
\begin{Theorem}{\rm\cite[Theorem 3.1]{He13}\cite[Theorem 4]{HSR03}}
	Let $\Phi:\mathcal{T}(H)\to\mathcal{T}(K)$ be a channel. The following are equivalent.
	\begin{itemize}
		\item [(a)] $\Phi$ is a SEB channel.
		\item [(b)] $\Phi$ has the form
		\begin{equation}\label{holevo-SEB}
			\Phi(X)=\sum_{k=1}^\infty R_k\operatorname{Tr}(F_k X)
		\end{equation} 
		for some states $R_k\in S(K)$ and positive operators $F_k\in B(H)^+$ such that \linebreak $\sum_{k=1}^\infty F_k=I_H$.
		\item [(c)] $\Phi$ has an operator-sum representation by rank-one operations, i.e.
		\begin{equation}\label{rank1rep}
			\Phi(X)=\sum_{k=1}^\infty u_kv_k^*X v_ku_k^*=\sum_{k=1}^\infty (v_k^*X v_k)u_ku_k^*,
		\end{equation}
		for some unit vectors $\{u_k\}\subseteq K$ and vectors $\{v_k\}\subseteq H$ such that $\sum_{k=1}^\infty v_kv_k^*=I_H$.
	\end{itemize}
\end{Theorem}

It is shown in \cite[Lemma 2.9]{LD15} that a channel $\Phi$ is SEB if the output state $(I_H\otimes\Phi)(\rho_0)$ corresponding to the maximally entangled state $\rho_0$ in $S(H\otimes H)$ is countably separable. However, determining whether a state is separable (or countably separable) is indeed very challenging. Therefore, it is interesting to find some conditions ensuring a channel is SEB other than the separability of the output states. 

In finite-dimensional systems, if the range of a completely positive map is commutative, then it is EB, see \cite[Lemma III.1]{RJP18} and \cite[Corrollary 3]{Stormer08}. 
However, the validity of these assertions may not extend seamlessly to infinite-dimensional systems, particularly concerning SEB channels. 
Our contribution in Theorem \ref{Thm2.2} establishes that the commutativity of the range serves as a sufficient condition to guarantee the SEB nature of a channel.
Conversely, as demonstrated in Proposition \ref{Prop2.3}, we affirm that any SEB channel can be dilated to the predual map of a completely positive map with commutative range.

In a recent study focusing on the structure of null spaces, the authors in \cite{KLOPR21} constructed some private subalgebras of certain classes of EB channels for finite-dimensional systems. In particular, they showed that every self-adjoint subspace of trace-zero matrices is the null space of an EB channel. The finding presented in Theorem \ref{thm3.2} expands upon this discovery, now encompassing applications to infinite-dimensional systems.

The article is structured as follows. In section \ref{section2}, we delve into various conditions that characterize a channel as SEB. In section \ref{section3}, we considered SEB channels whose null space is a given closed self-adjoint subspace of trace-zero operators for infinite-dimensional systems. Furthermore, we explore the attributes of the fixed point set of these channels concerning their representation via rank-one operations in \eqref{rank1rep}.

\section{Some characterizations of SEB channels}\label{section2}

As mentioned by Stormer \cite{Stormer08}, there is a natural duality between bounded linear maps on $B(H)$ and linear functionals on the tensor product $B(H)\, \hat{\otimes}\, \mathcal{T}(H)$. 
In the subsequent discussion, we leverage this duality within the context of SEB channels. It is demonstrated that a channel is SEB if and only if the corresponding linear functional satisfies a specific criterion associated with separability. We first revisit some notions outlined in \cite{Stormer08}.

Let $\Psi: B(H)\to B(H)$ be a bounded linear operator. Then a map
$\widetilde{\Psi}:B(H)\, \hat{\otimes}\, \mathcal{T}(H)\to\mathbb{C}$ given by
\begin{equation}\label{linearfunctional}
\widetilde{\Psi}(Y\otimes X)=\operatorname{Tr}(\Psi(Y)X^t) \quad \mbox{ for } X\in\mathcal{T}(H), Y\in B(H),
\end{equation}
defines a linear functional on the projective tensor product space of $B(H)$ and  $\mathcal{T}(H)$, where the transpose $X^t$ defined by $\left\langle e_i|X^t|e_j\right\rangle=\left\langle e_j|X|e_i\right\rangle$ with respect to a fixed orthonormal basis $\{e_i\}$ of $H$. A linear functional $\varphi$ on $B(H)\, \hat{\otimes}\, \mathcal{T}(H)$ is called \textit{separable} (see \cite{Stormer08}) if it belongs to the norm closure of elements of the form $\sum_{k=1}^{n}w_k\otimes\rho_k$ for some positive norm-one linear functional $w_k$ on $B(H)$ and positive linear functional  $\rho_k$ on $\mathcal{T}(H)$. The separability of the corresponding linear functional $\widetilde{\Psi}$ is equivalent to the following characterization of the operator $\Psi$.

\begin{Proposition}{\rm\cite[Theorem 2]{Stormer08}}\label{storm-thm}
Let $\Psi$ be a completely positive operator and let $\widetilde{\Psi}$ be the linear functiona defined in (\ref{linearfunctional}). Then $\widetilde{\Psi}$ is a separable linear functional if and only if $\Psi$ is a limit (in bounded-weak topology) of operators of the form $x\mapsto\sum_{i=1}^{n}w_i(x)b_i$ for some positive norm-one linear functional $w_i\in B(H)^*$ and positive operators $b_i\in B(H)^+$.
\end{Proposition}

By leveraging the aforementioned duality, we proceed to establish the separability of the associated linear functional in relation to a SEB channel.
\begin{Proposition}\label{prop2.2}
A quantum channel $\Phi:\mathcal{T}(H)\to\mathcal{T}(H)$ is SEB if and only if the linear functional $\widetilde{\Phi^*}:B(H)\,\hat{\otimes}\,\mathcal{T}(H)\to\mathbb{C}$ corresponding to the dual map $\Phi^*$ can be written of the form $\widetilde{\Phi^*}=\sum_{k=1}^{\infty}w_k\otimes\rho_k$ for some weak* continuous positive norm-one linear functional $w_k$ on $B(H)$ and positive linear operators $\rho_k$ in $B(H)^+$ with $\sum_{k=1}^{\infty}\rho_k=I_H$.
\end{Proposition} 
For convenience, we call a linear functional having the form as stated in the Proposition \ref{prop2.2} be {\it countably separable}.
The proof aligns with the argument presented in \cite[Theorem 2]{Stormer08} concerning SEB channels.
\begin{proof}
	 Suppose $\Phi:\mathcal{T}(H)\to\mathcal{T}(H)$ is a SEB channel.  Recall that a dual map of $\Phi$ is the linear map $\Phi^*: B(H)\to B(H)$ satisfying $\operatorname{Tr}(\Phi(X)Y)=\operatorname{Tr}(X\Phi^*(Y))$ for all $X\in\mathcal{T}(H)$ and $Y\in B(H)$. If $\Phi$ has the form \eqref{holevo-SEB} then its dual map satisfies
	\begin{equation}\label{dualmap}
		\Phi^*(Y)=\sum_{k=1}^\infty \operatorname{Tr}(R_kY)F_k,
	\end{equation}
	for some states $R_k\in\mathcal{T}(H)$ and positive operators $F_k\in B(H)^+$. For $Y\in B(H)$ and $X\in\mathcal{T}(H)$,
	$$\widetilde{\Phi^*}(Y\otimes X)=\operatorname{Tr}(\Phi^*(Y)X^t)=\sum_{k=1}^\infty \operatorname{Tr}(R_kY)\operatorname{Tr}(F_kX^t)=\sum_{k=1}^\infty w_k(Y)\rho_k(X),$$
	where 
	\begin{equation}\label{w-rho-operator}
		w_k(Y)=\operatorname{Tr}(R_kY)\quad \mbox{ and } \quad \rho_k(X)=\operatorname{Tr}(F_kX^t)
	\end{equation}
	 define a state in $B(H)^*$ and a positive linear functional on $\mathcal{T}(H)$, respectively. Then we have 
$$\sum_{k=1}^{\infty}\rho_k(X)=\operatorname{Tr}\left(\left(\sum_{k=1}^{\infty}F_k \right)X^t\right)=\operatorname{Tr}(X^t)=\operatorname{Tr}(X).$$
Thus, $\sum_{k=1}^{\infty}\rho_k=I_H$ and  $\widetilde{\Phi^*}$ is countably separable in $(B(H)\,\hat{\otimes}\,\mathcal{T}(H))^*$.
	
 Conversely, suppose $\widetilde{\Phi^*}$ is countably separable and having the form $\widetilde{\Phi^*}=\sum_{k=1}^{\infty}w_k\otimes\rho_k$ as in the proposition. 
 For each $k$, there exist states $R_k\in\mathcal{T}(H)$ and positive operators $F_k\in B(H)^+$ satisfying \eqref{w-rho-operator}. The condition $\sum_{k=1}^{\infty}\rho_k=I_H$ implies that $\sum_{k=1}^{\infty}F_k=I_H$. For each $X\in\mathcal{T}(H)$ and $Y\in B(H)$,
	$$\operatorname{Tr}\left(\Phi^*(Y)X^t\right)=\widetilde{\Phi^*}(Y\otimes X)=\sum_{k=1}^{\infty}w_k(Y)\rho_k(X)=\operatorname{Tr}\left(\sum_{k=1}^{\infty}\operatorname{Tr}(R_kY)F_kX^t\right).$$
	Then $\Phi^*$ has the form as in \eqref{dualmap}, and hence   $\Phi$ is a SEB channel.
\end{proof}

Herein, we present a sufficient condition adequate for identifying a channel as SEB, devoid of any reliance on the separability of states within the bipartite system. This condition asserts that if the range of a channel is commutative, then the channel qualifies as SEB.
\begin{Theorem}\label{Thm2.2}
Let $\{e_i\}$ be any orthonormal basis of a separable Hilbert space  $H$. Let $\Phi:\mathcal{T}(H)\to\mathcal{T}(H)$ be a completely positive and trace-preserving channel. If the range of $\Phi$ is commutative, then $\Phi$ is a SEB channel. Under this condition, for any set of scalars $\lambda_i>0$ with $\sum_{i=1}^\infty\lambda_i=1$, the weighted Choi state $\sigma$ of $\Phi$, defined by 
$$\sigma:=\sum_{i,j=1}^\infty\sqrt{\lambda_i\lambda_j}e_ie_j^*\otimes\Phi(e_ie_j^*),$$
is countably separable and can be written in the form
\begin{equation}\label{rholambda}
\sigma=\sum_{k=1}^{\infty}p_k\rho_k\otimes v_kv_k^*,
\end{equation}
for some orthonormal set of vectors $\{v_k\}$ in $H$, some states $\rho_k$ in $\mathcal{T}(H)$, and some scalars $p_k>0$ with $\sum_{k=1}^\infty p_k=1$. 
Moreover, the channel $\Phi$ can be expressed as $\Phi(X)=\sum_{k=1}^{\infty}R_k\operatorname{Tr}(XF_k)$, where
$R_k=(Uv_k)(Uv_k)^*$ for some unitary operator $U$ in $B(H)$, and operators $F_k\in B(H)^+$ with a matrix representation as $$F_k=\left[\frac{p_k}{\sqrt{\lambda_i\lambda_j}}\left\langle e_i,\rho_k e_j\right\rangle\right]_{1\leq i,j<\infty}.$$
\end{Theorem}

\begin{proof}
	We first note that if the range of a channel $\Phi:\mathcal{T}(H)\to\mathcal{T}(H)$ is commutative, then this range will consist of normal operators. Indeed, since $\Phi$ is positive, it maps self-adjoint operators to self-adjoint operators. Then for any operator $X$ in $\mathcal{T}(H)$ of the form $X=X_1+iX_2$ for some self-adjoint operators $X_1$ and $X_2$, we have $\Phi(X)^*=(\Phi(X_1)+i\Phi(X_2))^*=\Phi(X_1)-i\Phi(X_2)=\Phi(X^*)$. Thus, $\Phi$ preserves the involution. Since $\mathcal{T}(H)$ is a self-adjoint space, for each $X\in\mathcal{T}(H)$ we have $\Phi(X)\Phi(X)^*=\Phi(X)\Phi(X^*)=\Phi(X^*)\Phi(X)=\Phi(X)^*\Phi(X)$. Then $\Phi(X)$ is a normal operator for all $X\in\mathcal{T}(H)$.
	
	We now employ some techniques introduced in \cite{LD15} and \cite{GuoHou12}.
	From the above paragraph, $\Phi(e_ie_j^*), i,j=1,\ldots,\infty$, are mutually commuting, normal and trace-class (hence compact) operators. So they are simultaneously diagonalizable by some unitary operator $U$ in $B(H)$. Then
	\begin{equation}\label{eqChoimatrix}
		(I_H\otimes U^*)\sigma(I_H\otimes U)=\sum_{i,j=1}^{\infty}e_ie_j^*\otimes D_{ij},
	\end{equation}
	for some diagonal operators $D_{ij}$ in $B(H)$.
	Replacing $\Phi(\cdot)$ by $U\Phi(\cdot)U^*$, we can assume $\sigma$ has the form as in the right side of \eqref{eqChoimatrix}.
	Hence $\sigma$ can be written as 
	$$\sigma=\sum_kB_k\otimes P_k,$$
	for some rank-one orthogonal projections $P_k$ and positive operators $B_k$. On the other hand, since $\Phi$ is completely positive, the weighted Choi operator $\sigma$ defined as in \eqref{rholambda} is positive by \cite[Theorem 1.4]{LD15}. It is easy to check that $\sigma$ has trace-one, hence it is a state in $B(H\otimes H)$. Thus $\sigma$ has the form as \eqref{rholambda} for some orthonormal set of vectors $\{v_k\}$ in $H$, some states $\{\rho_k\}$ in $\mathcal{T}(H)$ and some scalars $p_k\geq0$ with $\sum_{k=1}^\infty p_k=1$.
	
	For each state $X=\sum_{i,j}e_{ij}\otimes X_{ij}\in S(H\otimes H)$, define the partial trace of $X$ on the second component by $\operatorname{Tr}_2(X)=\sum_{i,j}e_{ij}\operatorname{Tr}(X_{ij})$. Since $\Phi$ is trace preserving, taking the partial trace over the second component of \eqref{rholambda}, we arrive at
	\begin{equation}\label{eq1111}
		\sum_{i=1}^\infty p_i\rho_i=\sum_{i,j}\lambda_ie_ie_j^*\operatorname{Tr}(\Phi(e_ie_j^*))=\sum_{i=1}^\infty\lambda_ie_ie_i^*.
	\end{equation}
	From \eqref{rholambda}, for $x,y\in H$,
	$$\sqrt{\lambda_i\lambda_j}\left\langle\Phi(e_ie_j^*)x,y\right\rangle=\left\langle \sigma(e_i\otimes x),(e_j\otimes y)\right\rangle=\sum_{k=1}^{\infty}\left\langle e_i,p_k\rho_ke_j\right\rangle\left\langle v_kv_k^*x,y\right\rangle.$$
	Then
	\begin{equation}\label{eq55}
		\Phi(e_ie_j^*)=\sum_{k=1}^{\infty}\frac{p_k}{\sqrt{\lambda_i\lambda_j}}\left\langle e_i,\rho_ke_j\right\rangle v_kv_k^*.
	\end{equation} 
		Following the proof of \cite[Lemma 2.9]{LD15}, we let $F_k$ be the Schur product as
	\begin{equation*}
		F_k=	p_k\begin{pmatrix}
			\frac{1}{\lambda_1} & \sqrt{\frac{1}{\lambda_1\lambda_2}}& \vdots \\
			\sqrt{\frac{1}{\lambda_2\lambda_1}} &  	\frac{1}{\lambda_2}& \vdots \\
			\cdots & \cdots & \ddots
		\end{pmatrix}\circ
		\begin{pmatrix}
			(\rho_k)_{11} & (\rho_k)_{12}& \vdots \\
			(\rho_k)_{21} &  (\rho_k)_{22}& \vdots \\
			\cdots & \cdots & \ddots
		\end{pmatrix}
		=\sum_{i,j=1}^{\infty}\frac{p_k}{\sqrt{\lambda_i\lambda_j}}(\rho_k)_{ij}e_ie_j^*,
	\end{equation*}
	where $(\rho_k)_{ij}=\left\langle e_i,\rho_ke_j\right\rangle$.
	We have $F_k\geq 0$ and $\sum_{k=1}^\infty F_k=I_H$ by \eqref{eq1111}. The map $\Psi(X)=\sum_{k=1}^{\infty}\operatorname{Tr}(XF_k)v_kv_k^*$ is well-defined and completely positive on $\mathcal{T}(H)$, see \cite[lemma 2.5]{LD15}. Equation \eqref{eq55} ensures that $\Psi(e_ie_j^*)=\Phi(e_ie_j^*)$ for all $i,j$. Hence $\Phi=\Psi$ on $\mathcal{T}(H)$ from the continuity in trace norm of these maps, see the proof of \cite[Lemma 2.9]{LD15} for details. Thus, $\Phi$ has the desired form.
\end{proof}

\begin{Remark}
Theorem \ref{Thm2.2} shares a connection with a result by Stormer \cite{Stormer08}. 
Specifically, Corollary 3 and Theorem 2 (along with Proposition \ref{storm-thm}) in \cite{Stormer08} imply that a completely positive operator with commutative range serves as a limit of \textit{entanglement breaking maps}. These maps bear similarities to the dual maps in \eqref{holevo-SEB}, albeit with a finite count of summands. In Theorem \ref{Thm2.2}, we demonstrate that a completely positive and trace-preserving channel with commutative range embodies a form characterized by an infinite summation akin to \eqref{holevo-SEB}.
\end{Remark}

The subsequent demonstration proves an anticipation in \cite{BaHo24}. It establishes that any strongly entanglement breaking channel can be dilated to the predual of a positive map with commutative range. Notably, the predual of an operator $\Psi: B(K) \to B(E)$ is represented by the operator $\Psi_*: \mathcal{T}(E) \to \mathcal{T}(K)$, which satisfies the following relationship:
$$\operatorname{Tr}\left(\Psi(Y)X\right)=\operatorname{Tr}\left(Y\Psi_*(X)\right),\quad \forall X\in\mathcal{T}(E), Y\in B(K).$$
\begin{Proposition}\label{Prop2.3}
If a channel $\Phi: \mathcal{T}(H)\to \mathcal{T}(K)$  is SEB, then there exist a Hilbert space $E$, an isometry $U: H\to E$, and a positive linear map $\Psi:B(K)\to B(E)$ with $\operatorname{range}(\Psi)$ being commutative such that $\Phi(X)=\Psi_*(UXU^*)$ for all $X\in\mathcal{T}(H)$.
\end{Proposition}

\begin{proof}
	We follow the arguments in \cite{BaHo24} which was for finite-dimensional systems. Let $\Phi^*: B(K)\to B(H)$ be a dual of $\Phi$ with a form
	\begin{equation*}
		\Phi^*(Y)=\sum_{k=1}^\infty \operatorname{Tr}(R_kY)F_k,
	\end{equation*}
	for some states $R_k\in\mathcal{T}(K)$ and positive operators $F_k\in B(H)^+$ satisfying $\sum_{k=1}^\infty F_k=I_H$. Let $l^\infty$ be the space of bounded sequences. Then $l^\infty$ is an unital commutative $C^*$-algebra with pointwise sums and products,  which can be identified with $C(\beta\mathbb{N})$, the space of continuous functions on the Stone-Cech compactification of $\mathbb{N}$.
	
	We write $\Phi^*$ into the composition $\Phi^*=\eta\circ\gamma$ of unital maps $\gamma: B(K)\to l^\infty$ by $\gamma(Y)=\{\operatorname{Tr}(YR_k)\}_k$, and $\eta: l^\infty\to B(H)$ by $\eta(\{a_i\}_i)=\sum_{i=1}^{\infty}a_iF_i$. The positive linear map $\eta$ is well-defined by conditions on $\{F_i\}_i$. Moreover, it is completely positive since $l^\infty$ is a commutative $C^*$-algebra, see \cite[Theorem 3.11]{Paul02}. By the Stinespring dilation theorem, there exist a Hilbert space $E$, an isometry $U: H\to E$ and a $\ast$-homomorphism $\pi: l^\infty\to B(E)$ such that $\eta(X)=U^*\pi(X)U$. Then for each $Y \in B(K)$, we have $$\Phi^*(Y)=\eta\circ\gamma(Y)=U^*(\pi\circ\gamma)(Y)U=U^*\Psi(Y)U,$$ where $\Psi=\pi\circ\gamma: B(K)\to B(E)$ is a weak*-weak* continuous positive mapping with commutative range since $\gamma$ is so. For each $Y\in B(K)$ and $X\in\mathcal{T}(H)$,
	$$\operatorname{Tr}(Y\Phi(X))=\operatorname{Tr}(\Phi^*(Y)X)=\operatorname{Tr}(U^*\Psi(Y)UX)=\operatorname{Tr}(\Psi(Y)UXU^*)=\operatorname{Tr}(Y\Psi_*(UXU^*).$$
	Hence, $\Phi(X)=\Psi_*(UXU^*)$ as asserted.
\end{proof}


\section{Null space, fixed point and multiplicative domain}\label{section3}

In \cite{KLOPR21}, the authors construct an entanglement breaking channel that vanishes on a given subspace of trace-zero matrices as follows.
\begin{Proposition}{\rm\cite[Proposition 3.1]{KLOPR21}}\label{prop1}
	Let $H$ be a finite-dimensional Hilbert space. Let $\mathcal{N}$ be a self-adjoint subspace of trace-zero operators in $B(H)$. Then there is an entanglement breaking channel $\Phi$ on $B(H)$ such that $\mathcal{N}=\{X\in B(H):\Phi(X)=0\}$. 
\end{Proposition}
We demonstrate the validity of this assertion in the context of infinite-dimensional systems. Here, $\mathcal{T}(H)_0$ represents the subspace of trace-zero operators within $\mathcal{T}(H)$, and the null space of a channel $\Phi$ on $\mathcal{T}(H)$ is denoted by $\operatorname{null}(\Phi): = \{X \in \mathcal{T}(H) : \Phi(X) = 0\}$. Given that every channel operating on $\mathcal{T}(H)$ is continuous in the trace-norm topology, see \cite[Lemma 2.3]{LD15}, its null space must be closed within this topology.
\begin{Theorem}\label{thm3.2}
Let $H$ be a separable infinite-dimensional Hilbert space. Let $\mathcal{N}\subseteq\mathcal{T}(H)_0$ be a self-adjoint and closed (in trace-norm) subspace of trace-zero operators on $H$. Then there is a SEB channel $\Phi$ on $\mathcal{T}(H)$ such that $\operatorname{null}(\Phi)=\mathcal{N}$.
\end{Theorem}

\begin{proof}
Let $\mathcal{N}^\perp$ be a subspace of $B(H)$ defined by
$$\mathcal{N}^\perp=\{Y\in B(H): \operatorname{Tr}(XY)=0 \quad \hbox{for all } X\in\mathcal{N}\}.$$
The identity operator $I_H$ is in $\mathcal{N}^\perp$ since $\mathcal{N}\subseteq \mathcal{T}(H)_0$. Let $\{e_k\}_k$ be an orthonormal basis of $H$. For each $Y\in\mathcal{N}^\perp$,
$$0=\operatorname{Tr}(YX)=\sum_{i=1}^{\infty}\left\langle e_k,YXe_k\right\rangle=\sum_{k=1}^{\infty}\left\langle X^*Y^*e_k,e_k\right\rangle=\operatorname{Tr}(X^*Y^*) \quad \hbox{for all } X\in\mathcal{N}.$$
Since $\mathcal{N}$ is self-adjoint, i.e. $\mathcal{N}=\{X^*: X\in\mathcal{N}\}$, we have $Y^*\in\mathcal{N}^\perp$. Thus $\mathcal{N}^\perp$ is a self-adjoint subspace of $B(H)$. We follow arguments in \cite[Lemma 2]{Yashin20} to define a countable subset whose linear span is dense (in weak* topology) in $\mathcal{N}^\perp$.

Recall the duality between $(B(H),\|\cdot\|)$ and $(\mathcal{T}(H),\|\cdot\|_1)$ is defined by $\left\langle X,Y\right\rangle=\operatorname{Tr}(XY)$ for all $X\in\mathcal{T}(H)$ and $Y\in B(H)$. The weak* topology on $B(H)$ is the topology defined by the family of semi-norms $\{q_X: X\in\mathcal{T}(H)\}$, where $q_X(Y)=|\operatorname{Tr}(XY)|$ for all $Y\in B(H)$. In addition, since the Hilbert space $H$ is separable, the trace-class $\mathcal{T}(H)$ is a separable Banach space with respect to the trace-norm. This implies that the set
$$B:=\left\{Y\in \mathcal{N}^\perp: \|Y-I_H\|\leq\frac{1}{2}\right\}$$
is weak* compact and metrizable. Hence, there exists a countable subset of self-adjoint operators $\{I_H, \tilde{F}_2, \tilde{F}_3,\ldots \}$ which is dense in $B$ in the weak* topology.  Since $\|I_H-\tilde{F}_k\|\leq\frac{1}{2}$, each self-adjoint operator $\tilde{F}_k$ is positive with the operator-norm $\|\tilde{F}_k\|\leq 2$. Let $F_k=\frac{1}{2^k}\tilde{F}_k$ for $k\geq 2$ and $F_1=I_H-\sum_{k=2}^{\infty}F_k$. Then $\|I_H-F_1\|=\|\sum_{k=2}^{\infty}F_k\|\leq \sum_{k=2}^{\infty}\frac{1}{2^{k-1}}=1$. Hence $\{F_k\}_{k=1}^\infty$ is a set of self-adjoint positive operators in $B(H)$ with $\sum_{k=1}^{\infty}F_k=I_H$ and its linear span is weak* dense in $\mathcal{N}^\perp$.

Let $R_k=e_ke_k^*$ be states in $B(H)$. Consider a SEB channel $\Phi:\mathcal{T}(H)\to\mathcal{T}(H)$ defined by
\begin{equation*}
	\Phi(X)=\sum_{k=1}^\infty R_k\operatorname{Tr}(F_k X).
\end{equation*} 
Then $X\in\operatorname{null}(\Phi)$ if and only if $\operatorname{Tr}(XF_k)=0$ for all $k$. This happens exactly for $X\in(\mathcal{N}^\perp)^\perp=\mathcal{N}$ since $\mathcal{N}$ is a closed subspace in trace-norm. Hence $\operatorname{null}(\Phi)=\mathcal{N}$ as asserted.
\end{proof}

Throughout the remainder of this section, we delve into examining the fixed point set and the multiplicative domain of SEB channels by leveraging the operator-sum representation through rank-one operations as in \eqref{rank1rep}. It is worth recalling that the fixed point set of a map $\Psi: Z \to Z$ is defined as $\operatorname{Fix}(\Psi) := \{T \in Z : \Psi(T) = T\}$. For the details on the characterization of fixed point sets of quantum operations, refer to \cite{AGG02, Li11}.

Let $\Phi:\mathcal{T}(H)\to\mathcal{T}(H)$ be a SEB channel defined as in \eqref{rank1rep} by
\begin{equation}\label{strongEB}
	\Phi(X)=\sum_{k=1}^\infty u_kv_k^*X v_ku_k^*=\sum_{k=1}^{\infty}E_kXE_k^*,
\end{equation}
where $E_k=u_kv_k^*$ for $u_k, v_k\in H$ satisfying $\|u_k\|=1$ for all $k$ and $\sum_{k=1}^{\infty}v_kv_k^*=I_H$. Then $\Phi^*:B(H)\to B(H)$ is an unital completely positive operator with rank-one Kraus operations $\{E_k^*=v_ku_k^*\}$.

Suppose a projection $P\in \operatorname{Fix}(\Phi^*)$. Then $\sum_{k=1}^{\infty}E_k^*PE_k=P$. By multiplying from both sides of this equation with $I_H-P$, we derive $PE_k^*(I_H-P)=0$ for all $k$. Similarly, by multiplying both sides of equation $\Phi^*(I_H-P)=I_H-P$ by $P$, we get $(I_H-P)E_k^*P=0$. Hence, $PE_k^*=E_k^*P$ for all $k$. Taking the adjoint on both sides, we deduce  $[P,E_k]=PE_k-E_kP=0$. Substitute $E_k=u_kv_k^*$ we obtain
$$Pv_k=Pv_k(u_k^*u_k)=P(v_ku_k^*)u_k=(v_ku_k^*)(Pu_k)=v_k(u_k^*Pu_k).$$
Thus, $Pv_k=\lambda_kv_k$ for some scalar $\lambda_k$. Similarly, $Pu_k=\beta_ku_k$ for some scalar $\beta_k$. Hence, $v_k$ is an eigenvector of $P$ for all $k$. Let $P,Q$ be projections in $\operatorname{Fix}(\Phi^*)$. Then for each $k$, there are scalars $\lambda_k,\beta_k$ such that $Pv_k=\lambda_kv_k$ and $Qv_k=\beta_kv_k$. Thus,
\begin{equation}\label{commutative}
	\begin{array}{rl}
		 PQ&=PQ\displaystyle\left(\sum_{k=1}^{\infty}v_kv_k^*\right)=P\sum_{k=1}^{\infty}\beta_kv_kv_k^*=\sum_{k=1}^{\infty}\beta_k\lambda_kv_kv_k^*\\
		&=\displaystyle\sum_{k=1}^{\infty}\lambda_kQv_kv_k^*=\sum_{k=1}^{\infty}QPv_kv_k^*=QP.
	\end{array}
\end{equation}
Then $P$ and $Q$ are commutative. Let $\mathcal{A}=\{A\in B(H):[A,E_k]=[A,E_k^*]=0\}$. Since $\mathcal{A}$ is spanned by its projections, $\mathcal{A}$ is a commutative von Neumann subalgebra of $\operatorname{Fix}(\Phi^*)$. For infinite-dimensional systems, the inclusion $\mathcal{A}\subseteq\operatorname{Fix}(\Phi^*)$ can be strict, see \cite{AGG02}. 
From the above observations, we have
\begin{Proposition}
Let $\Phi$ be a SEB channel on $\mathcal{T}(H)$ defined by \eqref{strongEB} with rank-one operations $\{E_k=u_kv_k^*\}$. A projection $P$ is a fixed point of $\Phi^*$ if and only if $[P,E_k]=0$ for all $k$. In this case, all vectors $u_k$ and $v_k$ are eigenvectors of $P$. The set of projections in $\operatorname{Fix}(\Phi^*)$ is commutative. 
\end{Proposition}

The study of fixed point sets is useful for the investigation of the multiplicative domain of a map. Recall that a \textit{multiplicative domain} of a map $\Phi$ on $\mathcal{T}(H)$ is the set (see e.g. \cite{RJP18}),
$$\mathcal{M}_\Phi=\{A\in\mathcal{T}(H):\Phi(AX)=\Phi(A)\Phi(X), \Phi(XA)=\Phi(X)\Phi(A) \hbox{ for all } X\in\mathcal{T}(H) \}.$$

\begin{Proposition}\label{prop-eig}
Let $\Phi$ be a SEB channel on $\mathcal{T}(H)$ defined by \eqref{strongEB} with rank-one operations $\{E_k=u_kv_k^*\}$. If a projection $P$ is in the multiplicative domain $\mathcal{M}_{\Phi}$, then $v_k$ are eigenvectors of $P$ for all $k$. The set of projections in $\mathcal{M}_{\Phi}$ is commutative.
\end{Proposition}
\begin{proof}
	It is known that if $P\in\mathcal{M}_\Phi$ then $P\in\operatorname{Fix}(\Phi^*\circ\Phi)$. Indeed, for each $X\in\mathcal{T}(H)$, we have
	$$\operatorname{Tr}(PX)=\operatorname{Tr}(\Phi(PX))=\operatorname{Tr}(\Phi(P)\Phi(X))=\operatorname{Tr}((\Phi^*\circ\Phi)(P)X).$$ Hence, $(\Phi^*\circ\Phi)(P)=P$. This implies $E_i^*E_jP=PE_i^*E_jP$ for all $i,j$. Substitute $E_i=u_iv_i^*$, we derive at $Pv_i=\lambda_iv_i$ for some scalar $\lambda_i$. The commutativity is similar as in \eqref{commutative}.  
\end{proof}

\section{Conclusions}
The paper establishes characterizations of strongly entanglement breaking (SEB) channels tailored for infinite-dimensional systems. It unveils properties akin to their finite-dimensional counterparts, as documented in works such as \cite{BaHo24, KLOPR21, RJP18}. Notably, the paper demonstrates that a channel is SEB if its range is commutative. Furthermore, it reveals that any SEB channel can be dilated to the predual of a positive map with commutative range. Additionally, a SEB channel is constructed with a null space that aligns with a specified self-adjoint and closed subspace of trace-zero operators. The discussion also delves into the commutativity of projections within the fixed point set and the multiplicative domain of these channels. The approach intertwines techniques from the finite-dimensional realm with expanded outcomes in operator theory tailored for infinite-dimensional systems, as expounded in \cite{GuoHou12, LD15, Yashin20}.

\section*{Acknowledgments}
Research of Sze was supported by a HK RGC grant
PolyU 15300121 and a PolyU research grant 4-ZZRN.
The HK RGC grant also supported the post-doctoral fellowship
of Muoi at the Hong Kong Polytechnic University.


\end{document}